\documentclass[10pt, xcolor]{article}
\usepackage{amsfonts}
\usepackage{amssymb,amsthm, amsmath,graphicx,bm,amsthm}
\usepackage{mathrsfs}
\usepackage{color,enumerate,lineno}
\usepackage{url,hyperref}
\usepackage{amsfonts}
\usepackage{amssymb,amsthm, amsmath,graphicx,bm,amsthm,mathtools}
\usepackage{color,enumerate}
\usepackage{url,lineno}
\usepackage{lscape}
\usepackage{multirow}
\usepackage{float}
\usepackage{fancyhdr}
\usepackage{emptypage}

\usepackage[latin1]{inputenc}

\usepackage[all]{xy}

\newtheorem{theorem}{Theorem}[section]

\newtheorem{proposition}[theorem]{Proposition}
\newtheorem{corollary}[theorem]{Corollary}
\newtheorem{lemma}[theorem]{Lemma}
\newtheorem{algorithm}[theorem]{Algorithm}
\theoremstyle{definition}

\newtheorem{example}[theorem]{Example}
\newtheorem{remark}[theorem]{Remark}

\def\CC{\mathscr{C}}

\def\e{{\mathrm e}}
\def\g{{\mathrm g}}

\def\m{{\mathrm m}}
\def\n{{\mathrm n}}
\def\t{{\mathrm t}}

\def\gcd{{\mathrm{gcd}}}

\def\A{{\mathrm{A}}}
\def\Ar{{\mathrm{Arf}}}
\def\F{{\mathrm F}}
\def\G{{\mathrm G}}

\def\S{{\mathrm S}}
\def\Max{\mathrm{Max}}
\def\SG{\mathrm{SG}}
\def\NN{{\mathrm N}}
\def\PF{{\mathrm{PF}}}
\def\msg{{\mathrm{ msg }}}

\def\Ap{{\mathrm{ Ap}}}
\def\max{{\mathrm{ max}}}
\def\MED{{\mathrm{MED}}}
\def\Cad{{\mathrm {Cad}}}

\def\msg{{\mathrm{ msg }}}

\def\Maximals{\mathrm{Maximals}_{\leq_S}}
\def\max{\mathrm{max}}

\def\min{\mathrm{min}}
\def\max{\mathrm{max}}

\def\N{\mathbb{N}}

\def\Z{\mathbb{Z}}
\def\Q{\mathbb{Q}}

\def\rank{\mathrm{rank}\, }
\def\Ap{\mathrm{Ap}}
\def\int{\mathrm{int}}

\title{The set of Arf numerical semigroups with given Frobenius number}

\author{
	M. A. Moreno-Fr\'{\i}as \footnote{
		Dpto. de Matem\'aticas, Facultad de Ciencias,
		Universidad de C\'adiz, E-11510, Puerto Real  (C\'{a}diz, Spain).
		Partially supported by  Junta de Andaluc\'{\i}a group FQM-298, 
		Proyecto de Excelencia de la Junta de Andalucía ProyExcel\_00868, Proyecto de investigación del Plan Propio--UCA 2022-2023 (PR2022-011) and Proyecto de investigación del Plan Propio--UCA 2022-2023 (PR2022-004).
		E-mail: mariangeles.moreno@uca.es.}
	\and
	J. C. Rosales \footnote{
		Dpto. de \'Algebra, Facultad de Ciencias, Universidad de Granada,
		E-18071, Granada. (Spain).
		Partially supported by  Junta de Andaluc\'{\i}a group FQM-343,
		Proyecto de Excelencia de la Junta de Andalucía ProyExcel\_00868 and Proyecto de investigación del Plan Propio--UCA 2022-2023 (PR2022-011).
		E-mail: jrosales@ugr.es.}
}

\date{}

\begin{document}
 
\maketitle

\begin{abstract}
	
In this work we will show that if $F$ is  a positive integer, then the set $\Ar(F)=\{S\mid S \mbox{ is an Arf numerical semigroup with Frobenius number } F\}$ verifies the following conditions:   1) $\Delta(F)=\{0,F+1,\rightarrow\}$ is the minimum of $\Ar(F),$ 2) if $\{S, T\} \subseteq \Ar(F)$, then $S \cap  T \in \Ar(F),$ 3)  
 if $S \in \Ar(F),$ $S\neq \Delta(F)$ and $\m(S)=\min (S \backslash \{0\})$, then $S\backslash \{\m(S)\} \in \Ar(F)$. 
 
The previous results will be used to give an algorithm which calculates the set $\Ar(F).$  Also we will see that if $X\subseteq S\backslash \Delta(F)$ for some $S\in \Ar(F),$ then there is the  smallest element of $\Ar(F)$ containing $X.$


   \smallskip
    {\small \emph{MSC-class:} 20M14 (Primary),  11D07, 13H10 (Secondary).}
\end{abstract}

\section{Introduction}

Denote by $\Z$ and $\N$  the set of integers and nonnegative integers, respectively.
A numerical semigroup is a
nonempty subset of $\N$ that is closed under addition, contains the zero element,
and whose complement in $\N$ is finite.

Let $\{n_1<\dots <\n_p\}\subseteq \N$ with $\gcd(n_1,\dots, n_p)=1.$ Then 
$$
\langle n_1,\dots, n_p \rangle =\left \{\sum_{i=1}^p \lambda_in_i \mid \{\lambda_1,\dots, \lambda_p\}\subseteq \N \right\}
$$
is a numerical semigroup and every numerical semigroup has this form (see \cite[Lemma 2.1]{libro}). The set ${n_1<\dots <n_p}$ is called {\it system of generators} of $S$, and we write $S=\langle  n_1,\dots, n_p \rangle.$ 
We say that a system of generators of a numerical semigroup is a {\it minimal system of generators} if none of its proper subsets generates the numerical semigroup. Every numerical semigroup has a unique minimal system of generators, which in addition is finite (see  \cite[Corollary 2.8]{libro}). The minimal system of generators of a numerical semigroup $S$, is denoted by $\msg(S).$ Its cardinality is called the {\it embedding dimension} and will be denoted by $\e(S).$

Given $S$  a numerical semigroup   the {\it multiplicity} of $S$, denoted by $\m(S),$ is the minimum of $S\backslash \{0\};$  the set of elements in $\N\backslash S$ is known as the set of {\it gaps} of $S$.  Its cardinality is called the {\it genus } of $S$, denoted by $\g(S),$ and
the greatest integer not in $S,$ denoted $\F(S),$ is known as its {\it Frobenius number}. These three invariants are very important in the theory numerical semigroups (see for instance \cite{alfonsin} and \cite{barucci} and the reference given there) and they will play a very important role in this work.

The called Frobenius problem (see \cite{alfonsin}) for numerical semigroups, lies in finding formulas to obtain the Frobenius number and the genus of a numerical semigroup from its minimal  system of generators. This problem was solved in \cite{sylvester} for numerical semigroups with embedding dimension two.  Nowadays, the problem is still open in the case of numerical semigroups with embedding dimension  greater than or equal to three. Furthemore, in this case the problem of computing the Frobenius number of a general numerical semigroup becomes NP-hard.

In the semigroup literature one can find a long list of works dedicated to the study of  one dimensional analytically irreducible domains via their value
semigroup (see  for instance \cite{bertin}, \cite{castellanos}, \cite{delorme}, \cite{kunz}, \cite{teissier} and \cite{watanabe}). One of the properties
studied for this kind of rings using this approach has been the Arf property. Based on  \cite{arf}, Lipman in \cite{lipman} introduces and motivates the study of Arf rings. The characterization of this rings in terms of their value semigroups gives rise to notion of Arf semigroup (referred to henceforth as $\A$-{\it semigroup}).

In order to collect common properties of some families of numerical semigroups, the concept of covariety was introduced in  \cite{covariedades}.  A  {\it covariety}  is a nonempty family $\CC$ of numerical semigroups that  fulfills  the following conditions:
\begin{enumerate}
	\item[1)]  $\CC$ has a minimum, denoted by   $\Delta(\CC)=\min(\CC).$ 
	\item[2)] If $\{S, T\} \subseteq \CC$, then $S \cap  T \in \CC$.
	\item[3)]  If $S \in \CC$  and $S \neq  \Delta(\CC)$, then $S \backslash \{\m(S)\} \in \CC$.
\end{enumerate}

In this work we study the set of $\A$-semigroup by using the techniques of covarieties.

 The structure of the paper is the following. In Section 2, we show that if $F$ is a positive integer, then the set $\Ar(F)=\{S\mid S \mbox{ is an } \A\mbox{-semigroup and }\F(S)=F\}$ is a covariety. This fact, together with the results that appear in \cite{covariedades}, will allow us, in Section 3,  to order the elements of the set $\Ar(F)$ in a rooted tree. These concepts  and results will be used, in Section 4, to present an algorithm which computes all the element of $\Ar(F).$

Following the terminology introduced in \cite{parametrizando}, we say that a $n$-sequence of  integers, $(x_1,\dots, x_n),$ is an {\it Arf sequence} (hereinafter, $\A$-sequence) provided that  
\begin{enumerate}
	\item[1)] $2\leq x_1 \leq \dots \leq x_n,$
	\item[2)] $x_{i+1}\in \{x_i,x_i+x_{i-1}, \dots, x_i+x_{i-1}+\dots+x_1,\rightarrow\}$ for all $i\in \{1,\dots,n-1\}$ (here $\rightarrow$ denotes that all integers larger than $x_i+x_{i-1}+\dots+x_1$ belong to the set).
\end{enumerate}
In Section 5, we will see that there is a one-to-one correspondence
between the $\A$-sequence and the $\A$-semigroups. Also we will study how the $\A$-sequence associated to the maximal elements of $\Ar(F)$ are.

We will say tha a set $X$ is an $\Ar(F)$-{\it set,} if it verifies the following conditions
\begin{enumerate}
	\item[1)] $X\cap \Delta(\Ar(F))=\emptyset.$
	\item[2)] There is $S\in \Ar(F)$ such that $X\subseteq S.$
	\end{enumerate}

In Section 6, we will prove that if $X$ is an $\Ar(F)$-set, then there is the least element on $\Ar(F)$ that contains $X$ and we will denote it by $\Ar(F)[X].$

If $X$ is an $\Ar(F)$-set and $S=\Ar(F)[X]$, we will say that $X$ is an $\Ar(F)$-{\it system of generators of } $S.$ We will demonstrate that every element $S$ of  $\Ar(F)$, admits a unique minimal system of generators, denoted by $\Ar(F)\msg(S).$ The cardinality of $\Ar(F)\msg(S)$ is call the $\Ar(F)$-{\it rank} of $S$ and we denote it by $\Ar(F)_{\rank}(S).$ We finish Section 6, by characterizing the elements $S$ of $\Ar(F)$ with $\Ar(F)_{\rank}(S)$ equal to one.

\section{Basic results}
The following result appears in \cite[Proposition 3.10]{libro}.
\begin{proposition}\label{proposition1}
If $S$ is a numerical semigroup, then $\e(S)\leq \m(S).$
\end{proposition}

A numerical semigroup S is said to have {\it maximal embedding
dimension} (from now on $\MED$-{\it semigroup}) if $\e(S) = \m(S).$ 

In \cite[Proposition 3.12]{libro}  the following result it is shown.

\begin{proposition}\label{proposition2}
Let $S$ be a numerical semigroup. Then the following conditions are equivalent.
\begin{enumerate}
	\item[1)] $S$ is a $\MED$-semigroup.
	\item[2)] $x+y-\m(S)\in S$ for all $\{x,y\}\subseteq S\backslash \{0\}.$
	\item[3)] $\left(S\backslash \{0\} \right)+\{-\m(S)\}$ is a numerical semigroup.
\end{enumerate}
\end{proposition}

A numerical semigroup is an $\A$-{\it semigroup} if $x+y-z \in S$ for all $\{x,y,z\}\subseteq S$ such that $x\ge y \ge z.$

As a consequence from Proposition \ref{proposition2}, we have the following result.
\begin{corollary}\label{corolario3}
	Every $\A$-semigroup is a $\MED$-semigroup.
\end{corollary} 

The following result is deduced from \cite[Proposition 3.2]{libro} and it solves the Frobenius problem for $\MED$-semigroups and so for $\A$-semigroups.
\begin{proposition}\label{proposition4}
	Let $S$ be a $\MED$-semigroup such that $\msg(S)=\{n_1<n_2<\dots <n_e\}.$ Then $\F(S)=n_e-n_1$ and $\g(S)=\displaystyle \frac{1}{n_1}\left( n_2+\dots+n_e\right)-\frac{n_1-1}{2}.$
\end{proposition}

Our next aim in this section, will be to prove that if $F$ is a positive integer, then $\Ar(F)=\{S\mid S \mbox{ is an } \A\mbox{-semigroup and }\F(S)=F\}$ is a covariety.

The following result is well known and easy to prove.

\begin{lemma}\label{lemma5} Let $S$ and $T$ be numerical semigroups and $x\in S.$ Then the following hold:
	\begin{enumerate}
		\item[1)] $S\cap T$ is a numerical semigroup and $\F(S\cap T)=\max\{\F(S), \F(T)\}.$
		\item[2)] $S\backslash \{x\}$ is a numerical semigroup if and only if $x\in \msg(S).$
		\item[3)] $\m(S)=\min\left( \msg(S)\right).$
	\end{enumerate}	
\end{lemma}

The following result appears in  \cite[Proposition 3.22]{libro}.

\begin{lemma}\label{lemma6} 
	The finite intersection of $\A$-semigroups is again an $\A$-semigroup.
	\end{lemma}
\begin{proposition}\label{proposition7}
	If $F\in \N\backslash \{0\},$ then $\Ar(F)$ is a covariety. 	
\end{proposition}
\begin{proof}
	It is clear that $\Delta(F)=\{0,F+1,\rightarrow\}$ is the minimum of $\Ar(F).$ By applying Lemmas \ref{lemma5} and \ref{lemma6}, we deduce that if $\{S,T\}\subseteq \Ar(F),$ then $S\cap T \in \Ar(F).$ In order to conclude the proof, we will see that if $S\in \Ar(F)$ and $S\neq \Delta(F),$ then $S \backslash \{\m(S)\} \in \Ar(F).$ By Lemma \ref*{lemma5}, we know that $T=S \backslash \{\m(S)\}$ is a numerical semigroup. As  $S\neq \Delta(F),$ then $\F(T)=F.$
	Let $\{x,y,z\}\subseteq T$ such that $x\ge y \ge z.$ We distinguish two cases:
	\begin{itemize}
		\item If $z=0,$ then $x+y-z=x+y\in T.$
		\item If $z\neq 0,$ then  $x\ge y \ge z>\m(S).$ Therefore, $x+y-z\in S$ and $x+y-z>\m(S).$ Hence, $x+y-z\in T.$
	\end{itemize}		
	\end{proof}
Let $S$ be a numerical semigroup. We  define recursively the {\it associated sequence} to $S$, as follows:
\begin{itemize}
	\item $S_0=S,$
	\item $S_{n+1}=S_n\backslash \{\m(S_n)\}$ for all $n\in \N.$
\end{itemize}
Let $S$ be a numerical semigroup. We say that an element $s\in S$ is {\it small} if $s<\F(S).$ We denote by $\NN(S)$ the set of small elements of $S$. The cardinality of $\NN(S)$ is denoted by $\n(S).$

It is clear that the set $\{0,\dots, \F(S)\}$ is the disjoint union of the sets $\NN(S)$ and $\N\backslash S.$ Therefore, we have the following result.
\begin{lemma}\label{lemma8}
	If $S$ is a numerical semmigroup, then $\g(S)+\n(S)=\F(S)+1.$	
\end{lemma}
If $S$ is a numerical semigroup and  $\{S_n\}_{n\in \N}$ is its associated sequence, then the set  $\Cad(S)=\{S_0,S_1,\dots, S_{\n(S)-1}\}$ is called the {\it associated chain} to $S.$ Observe that $S_{\n(S)-1}=\Delta(\F(S))=\{0,\F(S)+1,\rightarrow\}.$

A characterization of the $\A$-semigroups is presented below. 

\begin{proposition}\label{proposition9}
	Let $S$ be a numerical semigroup. The
	following statements are equivalent:
	\begin{enumerate}
		\item[1)] $S$ is an $\A$-semigroup.
		\item[2)] If $T$ is an element of the associate sequence to $S,$ then $T$ is a $\MED$-semigroup.
		\item[3)] If $T \in \Cad(S),$ then $T$ is a $\MED$-semigroup.
	\end{enumerate}
\end{proposition}
\begin{proof}
	$1) \Longrightarrow 2).$ If $\{x,y\}\subseteq T\backslash \{0\}$ and $x\ge y \ge \m(T),$ then $x+y-\m(T)\in S$ and $x+y-\m(T)\ge \m(T).$ Therefore, $x+y-\m(T)\in T.$ By applying Proposition \ref{proposition2}, we have that $T$ is a $\MED$-semigroup.\\
	$2)\Longrightarrow 3).$ Trivial.\\	
		$3)\Longrightarrow 1).$ Let $\{x,y,z\}\subseteq S$ such that $x\ge y \ge z.$ We want to see that $x+y-z\in S.$  We distinguish three cases.	
		\begin{enumerate}
			\item If $z=0,$ then the result is true.
			\item If $z\ge \F(S)+1,$ then trivially the result is also true.
			 \item If $0<z< \F(S)+1,$ then there is $i\in \{0,1,\dots, \n(S)-1\}$ such that $z=\m(S_i).$ As $S_i$ is a $\MED$-semigroup and $\{x,y\}\subseteq S_i\backslash \{0\}$ then $x+y-z=x+y-\m(S_i)\in S_i.$ Therefore, $x+y-z\in S.$
		\end{enumerate}	
\end{proof}
\section{The tree associated to $\Ar(F)$}
A {\it graph} $G$ is a pair $(V,E)$ where $V$ is a nonempty set and
$E$ is a subset of $\{(u,v)\in V\times V \mid u\neq v\}$. The
elements of $V$ and $E$ are called {\it vertices} and {\it edges},
respectively. A {\it path (of
	length $n$)} connecting the vertices $x$ and $y$ of $G$ is a
sequence of different edges of the form $(v_0,v_1),
(v_1,v_2),\ldots,(v_{n-1},v_n)$ such that $v_0=x$ and $v_n=y$.

A graph $G$ is {\it a tree} if there exists a vertex $r$ (known as
{\it the root} of $G$) such that for any other vertex $x$ of $G$
there exists a unique path connecting $x$ and $r$. If  $(u,v)$ is an
edge of the tree $G$, we say that $u$ is a {\it child} of $v$.\\

Define the graph $\G(F)$ as follows:

\begin{itemize}
	\item the set of vertices of $\G(F)$ is $\Ar(F)$,
	\item $(S,T)\in \Ar(F) \times \Ar(F)$ is an edge of $\G(F)$ if and only if $T=S\backslash \{\m(S)\}.$
\end{itemize}

As a consequence from  \cite[Proposition 2.6]{covariedades} and Proposition \ref{proposition7}, we have the following result.
\begin{proposition}\label{proposition10}
	 $\G(F)$ is a tree with root $\Delta(F)=\{0,F+1, \rightarrow\}.$	
\end{proposition}

A tree can be  built recurrently starting from the root  and connecting, 
through an edge, the vertices already built with  their  children. Hence, it is very interesting to characterize the children of an arbitrary vertex of $\G(F).$ For this reason, we will introduce some concepts and results.

Following the notation introduced in \cite{JPAA}, an integer $z$ is a {\it pseudo-Frobenius number} of a numerical semigroup $S$ if $z\notin S$ and $z+s\in S$ for all $s\in S\backslash \{0\}.$  We denote by $\PF(S)$
the set formed by the  pseudo-Frobenius numbers of $S.$ The cardinality of $\PF(S)$ is an important invariant of $S$ (see \cite{froberg} and \cite{barucci}) called the {\it type} of $S,$  denoted by $\t(S).$\\

%

\begin{example}\label{example12}
	Let $S=\langle 5,7,9\rangle=\{0,5,7,9,10,12,14,\rightarrow\}.$ An easy computation shows that $\PF(S)=\{11,13\}.$ Hence, $\t(S)=2.$
	\end{example}
Given $S$  a numerical semigroup, we denote by  $\SG(S)=\{x\in \PF(S)\mid 2x \in S\}.$ Its elements will be called {\it special gaps} of $S.$

The following result appears in \cite[Proposition 4.33]{libro}.
\begin{lemma}\label{lemma13} Let $S$ be a numerical semigroup and $x\in \N\backslash S.$ Then $x\in \SG(S)$ if and only if $S \cup \{x\}$ is a numerical semigroup.	
\end{lemma}
The following result is deduced from \cite[Proposition 2.9]{covariedades}.
\begin{proposition}\label{proposition14}
	If $S\in \Ar(F),$ then the set formed by the children of $S$ in the tree $\G(F)$ is the set
	$$
	\{S\cup \{x\}\mid x\in \SG(S),\, x<\m(S) \mbox{ and }S\cup \{x\}\in \Ar(F)\}.
	$$
 \end{proposition}

Let $S\in \Ar(F)$ and $x\in \SG(S)$ such that $x<\m(S).$ Our next aim is to present an algorithmic procedure which allows us to determine if $S\cup \{x\}$ is  an element of $\Ar(F).$
\begin{lemma}\label{lemma15} Let $S\in \Ar(F)$ and $x\in \SG(S)$ such that $\F(S)\neq x<\m(S).$ Then $S\cup \{x\}\in \Ar(F)$ if and only if $S\cup \{x\}$ is a $\MED$-semigroup.	
\end{lemma}
\begin{proof}
		{\it Necessity.} It is a consequence from Corollary \ref{corolario3}.\\
			{\it Sufficiency.} By Lemma \ref{lemma13}, we know that $S\cup \{x\}$ is a numerical semigroup with Frobenius number $F.$ It is clear that $\Cad(S\cup \{x\})=\Cad(S)\cup \left\{S\cup \{x\}\right\}.$ The proof  concludes easily by using Proposition \ref{proposition9}.
\end{proof}
\begin{lemma}\label{lemma16}
	Let $S$ be a numerical semigroup and $x\in SG(S)$ such that $x<\m(S).$ Then the following conditions are equivalent.
	\begin{enumerate}
		\item[1)] $T=S\cup \{x\}$ is a $\MED$-semigroup.
		\item[2)] If $\{a,b\}\subseteq \msg(S),$ then $a+b-x\in S.$
	\end{enumerate}
	\end{lemma}
\begin{proof}
	$1) \Longrightarrow 2).$ If $\{a,b\}\subseteq \msg(S),$ then $\{a,b\}\subseteq T \backslash \{0\}.$ By applying Propoposition \ref{proposition2} we have that $a+b-x\in T.$ As $a+b-x>x$ then $a+b-x \in S.$
	
	$2) \Longrightarrow 1).$ To prove that $T$ is a $\MED$-semigroup, by using Proposition \ref{proposition2}, it is enough to see that if $\{p,q\}\subseteq T \backslash \{0\},$ then $p+q-x \in T.$ For it we will consider two cases:
	\begin{enumerate}
		\item If $x\in \{p,q\},$ then clearly the result is true.
		\item If $x\notin \{p,q\},$ then  $\{p,q\}\subseteq S\backslash \{0\}.$ Hence there are $\{a,b\}\subseteq \msg(S)$ and  $\{s,s'\}\subseteq S$ such that $p=a+s$ and $q=b+s'.$ Then $p+q-x=(a+b-x)+s+s'\in S\subseteq T.$
	\end{enumerate}

\end{proof}

Next we show the  announced algorithm.

\begin{algorithm}\label{algorithm17}\mbox{}\par
\end{algorithm}
\noindent\textsc{Input}: A numerical semigroup $S$ and $x\in \SG(S)$ such that $x<\m(S).$  \par
\noindent\textsc{Output}: $S\cup \{x\}$ is a $\MED$-semigroup or $S\cup \{x\}$ is not a $\MED$-semigroup.
\begin{enumerate}
	\item[(1)] If $a+b-x\in S$ for all $\{a,b\}\subseteq \msg(S),$ return $S\cup \{x\}$ is a $\MED$-semigroup.
	\item[(2)] Return $S\cup \{x\}$ is not a $\MED$-semigroup.
\end{enumerate}

We illustrate the usage of Algorithm \ref{algorithm17} with the following example.

\begin{example}\label{example18}
	Let $S=\langle 5,8,9,12\rangle =\{0,5,8,9,10,12,\rightarrow\}.$ Then $4\in \SG(S).$ As $5+5-4 \notin S,$ we can say that $S\cup \{4\}$ is not a $\MED$-semigroup. 
\end{example}
\section{Algorithm to compute $\Ar(F)$}
Let $S$ be a numerical semigroup and $n\in S\backslash \{0\}.$ Define the 
{\it Apéry set} of $n$ in $S$ (in honour of \cite{apery})  as 
$\Ap(S,n)=\{s\in S\mid s-n \notin S\}$. 

The following result is deduced from \cite[Lemma 2.4]{libro}.

\begin{lemma}\label{lemma19}
	Let $S$ be  a  numerical semigroup and $n\in S\backslash \{0\}.$ Then $\Ap(S,n)$ is a set with cardinality $n.$ Moreover, $\Ap(S,n)=\{0=w(0),w(1), \dots, w(n-1)\}$, where $w(i)$ is the least
	element of $S$ congruent with $i$ modulo $n$, for all $i\in
	\{0,\dots, n-1\}.$
\end{lemma}
From Proposition 3.1 of \cite{libro}, we can deduce the following result.
\begin{lemma}\label{lemma20}
	Let $S$ be a numerical semigroup. Then $S$ is a $\MED$-semigroup if and only if $\msg(S)=\left(\Ap(S,\m(S))\backslash \{0\} \right)\cup \{\m(S)\}.$
	\end{lemma}

Let $S$ be a numerical semigroup. Over $\Z$ we  define the following order relation: $a\leq_S b$ if $b-a \in S.$

The following result is Lemma 10 from \cite{JPAA}.
\begin{lemma}\label{lemma21} If $S$ is  a numerical semigroup and $n \in S\backslash
	\{0\}.$ Then
	$$
	\PF(S)=\{w-n\mid w \in \Maximals \Ap(S,n)\}.
	$$
\end{lemma}
The next lemma  has an immediate proof.

\begin{lemma}\label{lemma22}
	Let $S$ be a numerical semigroup, $n\in S\backslash \{0\}$ and $w \in \Ap(S,n).$ Then $w\in \Maximals(\Ap(S,n))$ if and only if $w+w'\notin
	\Ap(S,n)$ for all $w'\in \Ap(S,n)\backslash \{0\}. $\\	
\end{lemma}

The proof of the following result is very simple.

\begin{lemma}\label{lemma23} If $S$ is a numerical semigroup and $S\neq \N,$ then $$\SG(S)=\{x\in \PF(S)\mid 2x \notin \PF(S)\}.$$
\end{lemma}

\begin{remark}\label{remark24}
	Note that if $S$ is a numerical semigroup and we know $\Ap(S,n)$ for some $n\in S \backslash \{0\},$ as a consequence of Lemmas \ref{lemma21}, \ref{lemma22} and \ref{lemma23}, we can easily compute $\SG(S).$
\end{remark}
We will illustrate the content of the previous remark with an example. 
\begin{example}\label{example25} Let $S=\langle 5,7,9 \rangle.$  Then $\Ap(S,5)=\{0,7,9,16,18\}.$ By applying Lemma \ref{lemma22}, we have that $\Maximals \Ap(S,5)=\{16,18\}.$ We know, by Lemma \ref{lemma21},  that $\PF(S)=\{11,13\}.$ Finally,  Lemma \ref{lemma23} asserts that  $\SG(S)=\{11,13\}.$
	
\end{example}

The following result is straighforward to obtain. 

\begin{lemma}\label{lemma26}
	Let $S$ be a numerical semigroup, $n\in S \backslash \{0\}$ and $x\in \SG(S).$ Then $x+n\in \Ap(S,n).$ Furthemore, $\Ap\left( S\cup \{x\},n \right)=\left( \Ap(S,n)\backslash \{x+n\}\right)\cup \{x\}.$
\end{lemma}

\begin{remark}\label{27}
	Observe that as a consequence from Lemma \ref{lemma26}, if we know $\Ap(S,n),$ then we can effortlessly compute $\Ap(S\cup  \{x\},n).$ In particular, if $S\in \Ar(F),$  then  Lemma \ref{lemma26} allows us to calculate the set $\Ap(T,n)$ from $\Ap(S,n),$ for every child $T$ of $S$ in the tree $\G(F)$(see Proposition \ref{proposition14}).
\end{remark}

Next we illustrate the above remark with an example.
\begin{example}\label{example28}
	We consider again the numerical semigroup $S=\langle 5,7,9 \rangle.$  By Example \ref{example25} we know that    $\Ap(S,5)=\{0,7,9,16,18\}$ and $11\in \SG(S).$ By applying Lemma \ref{lemma26}, we have   $\Ap(S\cup \{11\},5)=\{0,7,9,11,18\}.$
\end{example}

\begin{proposition}\label{proposition29} Let $S$ be a numerical semigroup and $x\in \SG(S)$ such that $x<\m(S)$ and $S\cup \{x\}$ is a $\MED$-semigroup. Then the following conditions   hold.
	\begin{enumerate}
		\item[1)] For  every $i\in \{1,\dots, x-1\}$ there is $a\in \msg(S)$ such that $a\equiv i\, (\mbox{mod } x).$
		\item[2)] If $\alpha(i)=\min\{a\in \msg(S)\mid a\equiv i\, (\mbox{mod } x) \}$ for all $i\in \{1,\dots, x-1\},$ then $\msg(S\cup \{x\})=\{x,\alpha(1),\dots,\alpha(x-1)\}.$
	\end{enumerate}
\end{proposition}
\begin{proof}
	\begin{enumerate}
		\item[1)] As $S\cup \{x\}$ is a $\MED$-semigroup and $\m(S\cup \{x\})=x,$ then by Lemma \ref{lemma20}, we know that  if $\Ap(S\cup \{x\},x)=\{0,w(1),\dots, w(x-1)\},$ then $\msg(S\cup \{x\})=\{x,w(1),\dots,w(x-1)\}.$
		
		It is clear that $\msg(S\cup \{x\})\subseteq \msg(S)\cup \{x\}.$ Hence  $\{x,w(1),\dots,w(x-1)\}\subseteq \msg(S)\cup \{x\}.$ Then we deduce that for all $i\in \{1,\dots, x-1\}$ there is $a\in \msg(S)$ such that $a\equiv i\, (\mbox{mod }x).$
		\item[2)] As $\msg(S\cup \{x\})\subseteq \msg(S)\cup \{x\}$ and  $\msg(S\cup \{x\})=\{x,w(1),\dots,w(x-1)\},$ then  we deduce that $w(i)=\alpha(i)$ for every $i\in \{1,\dots, x-1\}.$ 
	\end{enumerate}
	
\end{proof}
The following example illustrates the above result.

\begin{example}\label{example30} Let $S=\langle 7,8,9,10,11,12,13\rangle,$ then $4\in \SG(S),$ $4<\m(S)=7$ and $S\cup \{4\}=\{0,4,7,\rightarrow\}=\langle4,7,9,10\rangle$ is a $\MED$-semigroup. Note that $\alpha(1)=9,$ $\alpha(2)=10$ and $\alpha(3)=7.$	
\end{example}
\begin{remark}\label{remark31} Note that as a consequence from Proposition \ref{proposition29}, if $S\in \Ar(F)$ and we know $\msg(S),$ then we can easily compute $\msg(T)$ for all child $T$ of $S$ in the tree $\G(F)$(see Proposition \ref{proposition29} and Lemma \ref{lemma15}).
	\end{remark}
We are now ready to show the  algorithm which gives title to this section. 
\begin{algorithm}\label{algorithm32}\mbox{}\par
\end{algorithm}
\noindent\textsc{Input}: A positive integer $F.$   \par
\noindent\textsc{Output}: $\Ar(F).$

\begin{enumerate}
	\item[(1)] $\Delta=\langle F+1,\dots, 2F+1 \rangle,$ $\Ar(F)=\{\Delta\}$ and  $B=\{\Delta\}.$ 
	\item[(2)] For every $S \in B,$  compute $\theta(S)=\{x\in \SG(S)\mid x<\m(S), x\neq F \mbox{ and } S\cup \{x\} \mbox { is a }\MED\mbox{-semigroup}\}.$
	\item[(3)] If $\displaystyle\bigcup_{S\in B}\theta(S)=\emptyset,$ then return $\Ar(F).$
	\item[(4)]  $C=\displaystyle\bigcup_{S\in B}\{S\cup \{x\}\mid x\in \theta(S)\}.$ 	
	\item[(5)] For all $S\in C$ compute $\msg(S).$
	\item[(6)]  $\Ar(F)= \Ar(F)\cup C,$ $B=C,$ and   go to Step $(2).$ 	
\end{enumerate}
Next we illustrate this algorithm with an example.

\begin{example}\label{example33}
	We are going to calculate the set $\Ar(5)$, by using Algorithm \ref{algorithm32}.
	\begin{itemize}
		\item $\Delta=\langle 6,7,8,9,10,11\rangle,$ $\Ar(5)=\{\Delta\}$ and $B=\{\Delta\}.$
		\item $\theta(\Delta)=\{3,4\}$ and $C=\{\Delta\cup \{3\}=\langle 3,7,8 \rangle, \Delta\cup \{4\}=\langle 4,6,7,9 \rangle\}.$
		\item $\Ar(5)=\{\Delta,\langle 3,7,8 \rangle,\langle 4,6,7,9 \rangle\}$ and $B=\{\langle 3,7,8 \rangle,\langle 4,6,7,9 \rangle\}.$
		\item $\theta(\langle 3,7,8 \rangle)=\emptyset,$ $\theta(\langle 4,6,7,9 \rangle)=\{2\}$ and $C=\{\langle 2,7 \rangle\}.$
		\item $\Ar(5)=\{\Delta,\langle 3,7,8 \rangle,\langle 4,6,7,9 \rangle, \langle 2,7 \rangle \}$ and $B=\{\langle 2,7 \rangle\}.$
		\item $\theta(\langle 2,7 \rangle)=\emptyset.$
		\item The algorithm returns $\Ar(5)=\{\Delta,\langle 3,7,8 \rangle,\langle 4,6,7,9 \rangle, \langle 2,7 \rangle \}.$
	\end{itemize}
\end{example}
\section{$\A$-sequences}
Let $n\in \N\backslash \{0\},$ we say that a  $n$-sequence of integers $(x_1,\dots,x_n)$ is an $\A$-{\it sequence} if the following conditions hold:
\begin{enumerate}
	\item[1)] $2\leq x_1 \leq x_2\leq \dots \leq x_n,$
	\item[2)] $x_{i+1}\in \{x_i, x_i+x_{i-1}, \dots,  x_i+x_{i-1}+\dots+x_1,\rightarrow\}.$
\end{enumerate}
The following result appears in \cite[Proposition 1]{parametrizando}.

\begin{proposition}\label{proposition34}
Let $S$ be a nonempty subset of $\N$ such that $S\neq \N.$ Then $S$ is an $\A$-semigroup if and only if  there exists an $\A$-sequence
$(x_1,\dots, x_n)$ such that 
$$
S=\{0, x_n, x_n+x_{n-1}, \dots,  x_n+x_{n-1}+\dots+x_1,\rightarrow\}.
$$
\end{proposition}

The numerical semigroup, $S,$ of the previous proposition is called the $\A$-{\it semigroup associated to the }$\A$-{\it sequence} $(x_1,\dots, x_n).$

As a consequence from Proposition \ref{proposition34}, we have the following result. 
\begin{corollary}\label{corollary35}
	Let $S=\{0= s_0<s_1<\dots<s_{\n(S)-1}<s_{\n(S)}=\F(S)+1, \rightarrow\}$ be a numerical semigroup. Then $S$ is an $\A$-semigroup if and only if $\left(s_{\n(S)}-s_{\n(S)-1},s_{\n(S)-1}-s_{\n(S)-2},\dots, s_1-s_0 \right)$ is an $\A$-sequence.	
\end{corollary}
The $\n(S)$-sequence of the previous corollary will be called {\it sequence associated to }$S.$

Observe that  Corollary \ref{corollary35}, provides us a procedure to determine if a numerical semigroup is an $\A$-semigroup.

Next we will illustrate the previous procedure with an example.

\begin{example}\label{example36}
	\begin{enumerate}
		\item By using Corollary \ref{corollary35}, we will see that the numerical semigroup  $S=\langle 4,6,21,23\rangle$ is an $\A$-semigroup. Indeed, if $S= \{0, 4,6,8,10,\\12,14,16,18,20,\rightarrow \},$ then $\F(S)=19.$ Hence, its associated sequence is $(20-18,18-16,16-14,14-12,12-10,10-8,8-6,6-4,4-0)=(2,2,2,2,2,2,2,2,4)$ which is, clearly an $\A$-sequence. Therefore, $S$ is an $\A$-semigroup. 	
		\item Let $S=\langle 4,17,18,23\rangle.$ Then we have  $S= \{0, 4,8,12,16,17, 18,20,\rightarrow \}$ and so $\F(S)=19.$ Therefore, its associated sequence is $(20-18,18-17,17-16,16-12,12-8,8-4,4-0)=(2,1,1,4,4,4,4),$ which is not  an $\A$-sequence. By using Corollary \ref{corollary35}, we can say that a $S$ is not an $\A$-semigroup. 			
	\end{enumerate}
	
\end{example}

Now, our aim is to study the $\A$-sequences associated to the maximal elements of $\Ar(F).$

\begin{lemma}\label{lemma37}If $\{S,T\}\subseteq \Ar(F),$ $S\subsetneq T$ and $x=\max(T\backslash S),$ then $S\cup \{x\}\in \Ar(F).$	
\end{lemma}
\begin{proof}
	As $2x\in T$ and $2x>x,$ then $2x\in S.$ If $s\in S\backslash \{0\},$ then $x+s\in T$ and $x+s>x.$ Hence $x+s\in S$ and consequently $S\cup \{x\}$ is a numerical semigroup. It is clear that $x <F$ and so $\F(S\cup \{x\})=F.$
	In order to conclude the proof, we will see that $S\cup \{x\}$ is an $A$-semigroup. For this we will show that if $\{a,b,c\}\subseteq S\cup \{x\}$ and $a\leq b \leq c,$ then $b+c-a\in S\cup \{x\}.$ We distinguish two cases.
	\begin{enumerate}
		\item If $\{a,b,c\}\subseteq S,$ then $b+c-a\in S\subseteq S\cup \{x\}.$
		\item  If $\{a,b,c\}\nsubseteq S,$ then $x\in \{a,b,c\}. $ We easily deduce that $b+c-a\in T$ and $b+c-a \ge x.$ Therefore, $b+c-a\in S\cup \{x\}.$
	\end{enumerate}
 \end{proof}
\begin{lemma}\label{lemma38}
Let $(x_1,\dots, x_n)$ be an $\A$-sequence and $a\in \N\backslash \{0,1\}.$ Then the following conditions holds:
\begin{enumerate}
	\item[1)] $(a, x_1-a,x_2,\dots, x_n)$ be an $\A$-sequence if and only if $a\leq \displaystyle \frac{x_1}{2}.$
	\item[2)] If $i\in \{2,\dots, n\},$ then $(x_1,\dots, x_{i-1},a, x_i-a, x_{i+1}, \dots, x_n)$ is an $\A$-sequence if and only if $a\in \{x_{i-1},x_{i-1}+x_{i-2},\dots, x_{i-1}+x_{i-2}+\dots,x_1, \rightarrow \}$ and $x_i\in \{2a,2a+x_{i-1}, \dots,2a+x_{i-1}+ \dots+x_1,\rightarrow \}.$
\end{enumerate}
\end{lemma}
\begin{proof}
	\begin{enumerate}
	 \item[1)] 	If $(a, x_1-a,x_2,\dots, x_n)$ is an $\A$-sequence, then $a\leq x_1-a$ and so  $a\leq \displaystyle \frac{x_1}{2}.$
	
	Conversely, if  $a\leq \displaystyle \frac{x_1}{2},$ then $a\leq x_1-a\leq x_2\leq \dots \leq x_n.$ It is clear that if $i\in \{2,\dots, n\},$ then $x_i\in \{x_{i-1},x_{i-1}+x_{i-2},\dots, x_{i-1}+\dots +x_1, \rightarrow \}$ and  so $x_i\in \{x_{i-1},x_{i-1}+x_{i-2},\dots, x_{i-1}+\dots+x_2+x_1-a+a, \rightarrow \}.$  Therefore $(a, x_1-a,x_2,\dots, x_n)$ is an $\A$-sequence.
	 \item[2)] If $(x_1,\dots, x_{i-1},a, x_i-a, x_{i+1}, \dots, x_n)$ is an $\A$-sequence, then $a\in \{x_{i-1},x_{i-1}+x_{i-2},\dots, x_{i-1}+x_{i-2}+\dots +x_1, \rightarrow \}$ and $x_i-a \in \{a, a+x_{i-1},\dots, a+x_{i-1}+\dots +x_1, \rightarrow \}$ and so $x_i \in \{2a, 2a+x_{i-1},\dots, 2a+x_{i-1}+\dots +x_1, \rightarrow \}.$
	 
	 Conversely, to prove that $(x_1,\dots, x_{i-1},a, x_i-a, x_{i+1}, \dots, x_n)$ is an $\A$-sequence it will enough to see that $x_i-a \in \{a, a+x_{i-1},\dots, a+x_{i-1}+\dots +x_1, \rightarrow \}$ and $a \in \{x_{i-1},\dots, x_{i-1}+\dots +x_1, \rightarrow \}.$ But this fact is deduced from hypothesis.
	\end{enumerate}
\end{proof}

We say that an $\A$-sequence $(x_1,\dots, x_n)$ {\it admits a proper refinement,} if there exist $a\in \N\backslash \{0,1\}$ and $i\in \{1,\dots,n\}$ such that $(x_1,\dots, x_{i-1},a, x_i-a, x_{i+1}, \dots, x_n)$ is an $\A$-sequence.

If $(x_1,\dots,x_n)$ is an $\A$-sequence, we denote by $\S(x_1,\dots,x_n)$ its associated $\A$-semigroup. Observe that by Proposition \ref{proposition34}, we know that  $\S(x_1,\dots,x_n)\in \Ar(F)$ if and only if $x_n+x_{n-1}+\dots+x_1=F+1.$ 

Denote by  $\mathcal{J}(F)=\{(x_1,\dots,x_k)\mid k\in \N\backslash \{0\},\, (x_1,\dots,x_k) \mbox{ is an $\A$-sequence}\\ \mbox{which does not admit proper refinements and }x_1+\dots+x_k=F+1\}.$ The set formed by all maximal elements of $\Ar(F)$ will be denoted by $\Max(\Ar(F)).$\\

The following result is  a consequence of Lemma \ref{lemma37}.

\begin{proposition}\label{proposition39}
	The correspondence $f:\mathcal{J}(F)\longrightarrow \Max(\Ar(F)), $ $f(x_1,\dots, x_k)=\S(x_1,\dots, x_k)$ is a biyective map. 
\end{proposition} 

Next we will see an example showing this result.

\begin{example}\label{example40}
	It is clear that the elements $(2,2,2,8),$ $(2,2,2,2,6),$ $(2,2,2,2,2,4)$ and $(2,2,2,2,2,2,2)$ are $\A$-sequences and each one is a refinement of the previous one.
 Moreover, $(2,2,2,2,2,2,2)$  does not admit proper refinements. Therefore,  $\S(2,2,2,2,2,2,2)\in \Max(\Ar(13)) $ and  $\S(2,2,2,8)\subseteq  \S(2,2,2,2,2,2,2).$
\end{example}
\section{$\Ar(F)$-system of generators}
If $X$ is an $\Ar(F)$-set, then we denote by $\Ar(F)[X]$ the intersection of all elements of $\Ar(F)$ containing $X.$ As $\Ar(F)$ is a finite set, then by applying  Proposition \ref{proposition7}, we have that the intersection of elements of $\Ar(F)$ is  again an element of $\Ar(F).$ Therefore, we have the following result.
\begin{proposition}\label{proposition41}
	Let $X$ be an $\Ar(F)$-set. Then $\Ar(F)[X]$ is the smallest element of $\Ar(F)$ containing $X.$
\end{proposition}

If $X$ is an $\Ar(F)$-set and $S=\Ar(F)[X],$ we will say that $X$ is an $\Ar(F)$-{\it system of generators} of $S.$ Besides, if $S\neq\Ar(F)[Y]$ for all $Y\subsetneq X,$ then $X$ is a {\it minimal} $\Ar(F)$-{\it system of generators} of $S.$

Our next aim is to prove that every element of   $\Ar(F)$ has a unique minimal $\Ar(F)$-system of generators. For this reason we present some necessary results.\\

The following result has an immediate proof.

\begin{lemma}\label{lemma42}
	Let $X$ and $Y$ be $\Ar(F)$-set such that $X\subseteq Y.$ Then $\Ar(F)[X]\subseteq \Ar(F)[Y].$
\end{lemma}
\begin{lemma}\label{lemma43} 
Let $S\in\Ar(F).$ Then $X=\{x\in \msg(S)\mid S\backslash \{x\}\in \Ar(F) \}$ is an 	$\Ar(F)$-set and $S=\Ar(F)[X].$	
\end{lemma}
\begin{proof}
	It is clear that $X$ is an $\Ar(F)$-set and $X\subseteq S.$ Hence $\Ar(F)\subseteq S.$  Let $T=\Ar(F)[X]$ and we suppose  that $T\subsetneq S.$ Then there is $a=\min(S\backslash T).$ So we deduce that $a\in  \msg(S)$ with $a<F.$ If $S=T\cup \{s_1>s_2>\dots >s_n\},$ then $s_n=a.$ Therefore, by applying repeatedly Lemma \ref{lemma37}, we have that $T$, $T\cup \{s_1\}, T\cup \{s_1,s_2\},\dots T\cup \{s_1,s_2,\dots,s_{n-1}\}$ are elements of $\Ar(F).$ Consequently, $S\backslash \{a\}=T\cup \{s_1,s_2,\dots,s_{n-1}\}$ is an element of $\Ar(F).$ Therefore, $\Ar(F)[X]\subseteq S\backslash \{a\}$ contradicting the fact that $a\in X.$ 
\end{proof}

\begin{proposition}\label{proposition44}
	Let $S\in\Ar(F).$ Then $X=\{x\in \msg(S)\mid S\backslash \{x\}\in \Ar(F) \}$ is the unique  minimal	$\Ar(F)$-system of generators of $S.$	
	\end{proposition}
\begin{proof}
	By Lemma \ref{lemma43}, we know that $X$ is an 	$\Ar(F)$-set and $S=\Ar(F)[X].$ To conclude the proof, we will see that if $Y$ is an 	$\Ar(F)$-set and $S=\Ar(F)[Y],$ then $X\subseteq Y.$ Indeed, if $X\nsubseteq Y,$ then there exists $x\in X\backslash Y.$ Therefore, $S\backslash \{x\}\in \Ar(F)$ and $Y\subseteq S\backslash \{x\}.$ Consequently, $\Ar(F)[Y]\subseteq S\backslash \{x\}$ which is absurd.
\end{proof}
Now we illustrate the previous result with an example.
\begin{example}\label{example45}
	It is clear that $S=\langle 6,8,10,31,33,35\rangle\in \Ar(29),$ $S\backslash \{6\}\in \Ar(29),$ $S\backslash \{8\}\in \Ar(29)$ and $S\backslash \{10\}\notin \Ar(29).$ By applying Proposition \ref{proposition44}, we assert that $\{6,8\}$ is the minimal $\Ar(29)$-system of generators of $S.$
\end{example}
As a consequence from Propositions \ref{proposition7} and  \ref{proposition44} we have the following result.

\begin{corollary}\label{corollary46}If $S\in\Ar(F)$ and $S\neq \Delta(F),$ then $\m(S)$ belongs to the minimnal $\Ar(F)$-system of generators of $S.$
\end{corollary}
If $S\in\Ar(F),$ then we denote by $\Ar(F)\msg(S)$ the minimal $\Ar(F)$-system of generators of $S.$ The cardinality of  $\Ar(F)\msg(S)$  is called the $\Ar(F)$-{\it rank} of $S$ and it will denote by $\Ar(F)_{\rank}(S).$

The following result is a consequence from Proposition \ref{proposition44} and Corollary \ref{corollary46}.

\begin{corollary}\label{corollary47}
	If $S\in\Ar(F),$ then the following assertions hold.
	\begin{enumerate}
		\item[1)] $\Ar(F)_{\rank}(S)\leq \e(S).$
		\item[2)] $\Ar(F)_{\rank}(S)=0$ if and only if $S=\Delta(F).$
		\item[3)] $\Ar(F)_{\rank}(S)=1$ if and only if $\Ar(F)\msg(S)=\{\m(S)\}.$		
	\end{enumerate}
\end{corollary}

Our next purpose is to describe the elements of $\Ar(F)$ with $\Ar(F)$-rank equal to one.

 For integers $a$ and $b,$ we say that $a$ {\it divides} $b$ if there exists an integer $c$ such that $b=ca,$ and we denote this by $a\mid b.$ Otherwise, $a$ {\it does not divide} $b$, and we denote this by $a\nmid b.$

\begin{lemma}\label{lemma48}
	If $\{m,F\}\subseteq \N,$ $2\leq m <F$ and $m\nmid F,$ then $S=\langle m \rangle \cup \{F+1,\rightarrow\}$ is an element of $\Ar(F)$ and $S=\Ar(F)[\{\m\}].$
\end{lemma} 
\begin{proof}
It can be easily shown that $S$ is an $\A$-semigroup, $\F(S)=F$ and every element of $\Ar(F)$ containing $\{m\},$ contains $S.$ Therefore, $S=\Ar(F)[\{\m\}].$
\end{proof}

\begin{proposition}\label{proposition49} Under the standing notation, the following conditions are equivalent.
	\begin{enumerate}
		\item[1)] $S\in \Ar(F)$ and $\Ar(F)_{\rank}(S)=1.$
		\item[2)] There is $m\in \N$ such that $2\leq m <F,$  $m\nmid F$ and $S=\langle m \rangle \cup \{F+1,\rightarrow\}.$
	\end{enumerate}
\end{proposition}
\begin{proof} $1)\Longrightarrow  2).$ By Corollary \ref{corollary47}, we know that $S=\Ar(F)[\{\m\}]$ and $S\neq \Delta(F).$ Then $\m(S)\in \N,$ $2\leq \m(S) <\F(S)$ and $\m(S)\nmid F.$ The result now follows from  Lemma \ref{lemma48}. 
	
$2) \Longrightarrow 1).$ Also, it is deduced from Lemma \ref{lemma48}.
	
 	\end{proof}
 
 	If $q\in \Q,$ then we denote  $\lfloor q \rfloor=\max \{z\in \Z\mid z\leq q\}.$ 
 	
 	\begin{corollary}\label{corollary50}
 		If $S\in \Ar(F)$ and $\Ar(F)_{\rank}(S)=1,$ then $\g(S)=F-\displaystyle \left\lfloor \frac{F}{\m(S)}\right\rfloor.$	
 	\end{corollary}
 	\begin{proof}
 		By Proposition \ref{proposition49}, we know that  $S=\langle \m(S) \rangle \cup \{F+1,\rightarrow\}.$ Hence, $S=\left\{0,\m(S),2\m(S),\dots,\left\lfloor \frac{F}{\m(S)}\right\rfloor \m(S),F+1,\rightarrow \right\}.$  Therefore  $\g(S)=F-\displaystyle \left\lfloor \frac{F}{\m(S)}\right\rfloor.$
 	\end{proof}
 
  We  explain in the following example, how we can use the previous result.
  
 	\begin{example}\label{example51} If $S=\Ar(17)[\{5\}],$ then  by applying Corollary \ref{corollary50}, we have that $\g(S)=17-\displaystyle \left\lfloor \frac{17}{5}\right\rfloor=17-3=14.$
 	\end{example}

 	Let $P_1,\dots, P_r$ be different positive prime intergers and $\{\alpha_1,\dots, \alpha_r\}\subseteq \N.$ We know that the number of positive divisors of $P_1^{\alpha_1}\dots P_r^{\alpha_r}$ is $(\alpha_1+1)\cdots (\alpha_r+1).$  Then, as a consequence of Proposition \ref{proposition49}, we have the following result. 
 	
 	\begin{corollary}\label{corollary52} Let $F$ be an integer such that $F\ge 2.$ If $F=P_1^{\alpha_1}\dots P_r^{\alpha_r}$ is the decomposition of $F$ into primes, then  $\{S\in \Ar(F)\mid  \Ar(F)_{\rank}(S)=1\}$ is a set with  cardinality $F-(\alpha_1+1)\cdots (\alpha_r+1).$	
 	\end{corollary}
 
 We finish the paper with an example where we apply the above corollary.
 	\begin{example}\label{example53}
 		As $360=2^3\cdot3^2\cdot 5^1$, then by applying Corollary \ref{corollary52}, we have that  the cardinality of  $\{S\in \Ar(360)\mid  \Ar(360)_{\rank}(S)=1\}$ is  $360-(3+1)\cdot(2+1)\cdot(1+1)=360-24=336.$	
 	\end{example}


\begin{thebibliography}{12}

\bibitem{apery}
\textsc{R. Ap\'{e}ry},
\newblock \emph{Sur les branches superlin\'{e}aires des courbes alg\'{e}briques,}
\newblock   C.R. Acad. Sci. Paris 222 (1946), 1198--2000.


 \bibitem{arf}
 \textsc{C. Arf},
 \newblock \emph{Une interprétation algébraique de la suite des ordres de multiplicité d'une branche algébrique,}
 \newblock   Proc. London Math. Soc. Ser.2,50(1949), 256--287.
 
 

\bibitem{barucci}
\textsc{V. Barucci, D. E. Dobbs and M. Fontana},
\newblock \emph{Maximality Properties in Numerical Semigroups and Applications to One-Dimensional Analitycally Irreducible Local Domains,}
\newblock   Memoirs Amer. Math. Soc. 598 (1997).

\bibitem{bertin}
\textsc{J. Bertin, P. Carbonne},
\newblock \emph{Semi-groupes d'entiers et application aux branches,}
\newblock   J. Algebra 49 (1977), 81--95.

\bibitem{castellanos}
\textsc{J. Castellanos},
\newblock \emph{A relation between the sequence of multiplicities and the semigroups of values of an algebroid curve,}
\newblock  J. Pure Appl. Algebra 43 (1986),119-127.

\bibitem{delorme}
\textsc{C. Delorme},
\newblock \emph{Sous-monoïdes d'intersection complète de $\N$,}
\newblock Ann. Scient. École Norm. Sup.(4)9(1976), 145--154 .



\bibitem{froberg}
\textsc{R. Fröber, G. Gottlieb and R. Häggkvist},
\newblock \emph{On numerical semigroups,}
\newblock   Semigroup Forum  35 (1987), 63--83.


\bibitem{parametrizando}
\textsc{P. A. Garc\'ia, B. A. Heredia, H. I. Karakas and  J. C.
Rosales},
\newblock \emph{Parametrizing Arf numerical semigroups,}
\newblock   J.  Algebra Appl. 16 (2017), 1750209 (31 pages)


\bibitem{kunz}
\textsc{E. Kunz},
\newblock \emph{The value-semigroup of a one-dimensional Gorenstein ring,}
\newblock  Proc. Amer. Math. Soc. 25(1973),
748--751.






\bibitem{lipman}
\textsc{J. Lipman},
\newblock \emph{Stable ideals and Arf rings,}
\newblock   Amer. J. Math. 93 (1971), 649--685.

\bibitem{covariedades}
\textsc{M. A. Moreno-Frías and J. C. Rosales},
\newblock \emph{The covariety of numerical semigroups with  fixed Frobenius number,}
\newblock 
https://doi.org/10.48550/arXiv.2302.09121



\bibitem{alfonsin}
\textsc{J. L. Ramírez Alfonsín},
\newblock \emph{The Diophantine Frobenius Problem,}
\newblock  Oxford Univ. Press, London (2005).

\bibitem{JPAA}
\textsc{J. C. Rosales and  M. B. Branco},
\newblock \emph{Numerical Semigroups that can be expressed as an intersection of symmetric numerical semigroups,}
\newblock   J. Pure Appl. Algebra  171 (2002), 303--314.



\bibitem{libro}
\textsc{J. C. Rosales and P. A. Garc\'ia-S\'anchez},
\newblock \emph{Numerical Semigroups,}
\newblock   Developments in Mathematics, Vol. 20, Springer, New York, 2009.

\bibitem{sylvester}
\textsc{J. J. Sylvester},
\newblock \emph{Mathematical question with their solutions,}
\newblock Educational Times 41  (1884), 21.


\bibitem{teissier}
\textsc{B. Teissier},
\newblock \emph{Appendice à ``Le probléme des modules pour le branches planes'',}
\newblock  cours donné par O. Zariski au Centre de Math. de L'École Polytechnique, Paris (1973).

\bibitem{watanabe}
\textsc{K. Watanabe},
\newblock \emph{Some examples of one dimensional Gorenstein domains,}
\newblock   Nagoya Math. J. 49 (1973), 101--109.



\end{thebibliography}
\end{document}